\documentclass[preprint]{elsarticle}

\usepackage{amsmath,amstext, amsthm,amsfonts,amssymb,amsthm}

\biboptions{sort&compress}

\newtheorem{thm}{Theorem}[section]
\newtheorem{cor}[thm]{Corollary}
\newtheorem{lem}[thm]{Lemma}

\theoremstyle{definition}
\newtheorem{defn}[thm]{Definition}
\theoremstyle{remark}

\numberwithin{equation}{section}

\newcommand{\norm}[1]{\left\Vert#1\right\Vert}

\newcommand{\set}[1]{\left\{#1\right\}}

\newcommand{\CP}{{}^cD_{q,a^+}}
\newcommand{\CM}{{}^cD_{q,{a}^-}}
\newcommand{\CZ}{{}^cD_{q,0^+}}

\newcommand{\ML}{\mathcal{L}}

\newcommand{\be}{\begin{equation}}
\newcommand{\ee}{\end{equation}}
\newcommand{\non}{\nonumber}
\newcommand{\bea}{\begin{eqnarray}}
\newcommand{\eea}{\end{eqnarray}}

\newcommand{\W}{W_{q,p,\alpha}}

\newcommand*{\Scale}[2][4]{\scalebox{#1}{$#2$}}

\begin{document}

\begin{frontmatter}



\title{On Fractional $q$-Sturm--Liouville Problems}


\author{Zeinab S.I. Mansour}

\address{Department of Mathematics, Faculty of Science, King Saud University, Riyadh}
 \ead{zsmansour@ksu.edu.sa, zeinabs98@hotmail.com}

\begin{abstract}
In this paper, we formulate a regular $q$-fractional Sturm--Liouville problem (qFSLP) which  includes  the left-sided Riemann--Liouville and  the right-sided Caputo $q$-fractional  derivatives  of the same  order $\alpha$, $\alpha\in (0,1)$.   The properties of the eigenvalues and the eigenfunctions  are investigated. A  $q$-fractional version of the Wronskian is defined and its relation to  the simplicity of the  eigenfunctions is verified.  We use the fixed point theorem to introduce a sufficient condition  on eigenvalues for the  existence and uniqueness of  the associated eigenfunctions  when $\alpha>1/2$. These results  are a generalization of  the  integer regular $q$-Sturm--Liouville problem introduced by Annaby and Mansour in\cite{AM1}. An example for a qFSLP whose eigenfunctions are little $q$-Jacobi polynomials is introduced.
\end{abstract}

\begin{keyword}
Left and right sided Riemann--Liouville and Caputo $q$-derivatives, eigenvalues and eigenfunctions, existence and uniqueness theorem, $q,\alpha$ Wronskian.


 \MSC      39A13\sep 26A33\sep 34L10.
\end{keyword}

\end{frontmatter}

\bigskip
\section{\bf{Introduction}}

In the joint paper of Sturm and Liouville~\cite{Sturm-Liouville},  they studied the problem
\be \label{SLE} -\frac{d}{dx}\left(p\frac{dy}{dx}\right)+r(x)y(x)=\lambda w y(x),\quad x\in [a,b], \ee
with certain boundary conditions at $a$ and $b$. Here, the functions $p$, $w$ are positive on $[a,b]$ and $r$ is a real valued function on $[a,b]$. They proved the existence
of non-zero solutions (eigenfunctions) only for special values of the parameter $\lambda$ which is called eigenvalues.
For a comprehensive study for the contribution of Sturm and Liouville to the theory, see~\cite{STL}. Recently, many mathematicians were interested in a fractional version of
\eqref{SLE}, i.e. when the derivative is replaced by a fractional derivative like  Riemann--Liouville derivative or Caputo derivative, see~\cite{KOM,Klimek-Agrawal,FSL1,FSL2,FSL3,FSL4}.
In~\cite{AM1} Annaby and Mansour introduced a $q$-version of \eqref{SLE}, i.e., when the derivative is replaced by Jackson $q$-derivative. Their results are applied and developed in different aspects, for example,
see~\cite{Lavagno,Abreu005,AnBuIs,AMI,Nemri-Fitouhi,Abreu007}. In this paper, we introduce a $q$-fractional Sturm--Liouville problem. The operator $ -\frac{d}{dx}\left(p\frac{dy}{dx}\right)$ in~\eqref{SLE} is self-adjoint operator because the adjoint of the operator $\frac{d}{dx}$  is $-\frac{d}{dx}$ in a certain subspace of $L^2(0,a)$. So, the negative sign in~\eqref{SLE} exists on purpose.  Therefore, the  fractional $q$-Sturm--Liouville problem  under consideration contains both  of the  left-sided Caputo $q$-fractional derivative  and
the right-sided Riemann--Liouville $q$-fractional derivatives  because they are adjoint operators in a certain Hilbert space.

This paper is organized as follows. Section 2 is on the $q$-fractional operators and their properties which we need in the sequel. In Section 3, we formulate the fractional $q$-Sturm--Liouville problem, we show that
the eigenvalues are real and the eigenfunctions associated to different eigenvalues are orthogonal in a certain Hilbert space. We use the fixed point theory to give a sufficient  condition on
  the parameter $\lambda$ to guarantee the existence  and uniqueness of the solution. We also impose a condition on the domain of the problem in order to prove the existence of  solution
for any $\lambda$.  The Wronskian associated with the problem is defined and a relation between its value   at zero and the simplicity of eigenfunctions is proved in Section 4. Finally, in Section 5, an example for a qFSLP whose eigenfunctions are little $q$-Jacobi polynomials is introduced.
Throughout this paper $q$ is a positive number less than 1. The set of non negative integers is denoted by $\mathbb{N}_0$, and the set of positive integers is denoted
by $\mathbb{N}$. For $t>0$,
\[A_{q,t}:=\set{tq^n\;:\;n\in\mathbb{N}_0},\;\;\;A_{q,t}^*:=A_{q,t}\cup\set{0},\]
and \[\mathcal{A}_{q,t}:=\set{\pm tq^n\;:\;n\in\mathbb{N}_0}.\]
When $t=1$, we simply use $A_{q}$, $A_q^*$, and $\mathcal{A}_q$ to denote $A_{q,1}$, $A_{q,1}^*$, and $\mathcal{A}_{q,1}$, respectively.    We follow~\cite{GR} for the definitions and notations of the $q$-shifted factorial, the $q$-gamma and $q$-beta functions, the basic hypergeometric series, and
Jackson $q$-difference operator and integrals. A set $A$ is called  a $q$-geometric set if $qx\in A$ whenever $x\in A$.
Let $X$ be a $q$-geometric set containing zero. A function $f$ defined on $X$ is called $q$-regular at zero if
\[\lim_{n\to\infty} f(xq^n)=f(0)\quad\mbox{for all}\;x\in X.\]
  Let $C(X)$  denote the space of all $q$-regular at zero functions defined  on
$X$ with values in $\mathbb{R}$. $C(X)$ associated with the norm  function
\[\norm{f} = \sup\set{|f(xq^n)|\; :\; x\in X,\; n\in\mathbb{N}_0},\]
is a normed space.
The $q$-integration by parts rule~\cite{AMbook} is
\be\label{qIR}\int_{a}^{b}f(x)D_qg(x)=f(x)g(x)|_{a}^{b}+\int_{a}^{b}D_qf(x)g(qx)\,d_qx,\; a,b\in X,\ee
and $f,\,g$ are $q$-regular at zero functions.

For $p>0$, and   $Y$ is $A_{q,t}$ or $A_{q,t}^*$, the space $L_q^p(Y)$ is the normed space
of all functions defined on $Y$ such that
\[\norm{f}_p:=\left(\int_{0}^{t}|f(u)|^p\,d_qu\right)^{1/p}<\infty. \]
If $p=2$, then $L_q^2(Y)$  associated with the inner product
\be\label{IP} \left<f,g\right>:=\int_{0}^{t}f(u)\overline{g(u)}\,d_qu\ee
is a Hilbert space. By a weighted $L_q^2(Y, w)$ space is the space of all functions  $f$ defined on $Y$
such that
\[\int_{0}^{t}|f(u)|^2 w(u)\,d_qu<\infty,\]
where $w$ is a positive function defined on $Y$. $L_q^2(Y, w)$ associated with the inner product
\[ \left<f,g\right>:=\int_{0}^{t}f(u)\overline{g(u)}w(u)\,d_qu\]
is a Hilbert space.
The space of all $q$-absolutely functions on $A_{q,t}^*$ is denoted by $\mathcal{A}C_q(A_{q,t}^*)$ and defined as the space of all $q$-regular
at zero functions $f$
satisfying \[\sum_{j=0}^{\infty}|f(uq^j)-f(uq^{j+1})|\leq K\;\mbox{for all}\; u\in A_{q,t}^{*}, \]
and $K$ is a constant depending on the function $f$, c.f.~\cite[Definition 4.3.1]{AMbook}. I.e.
\[\mathcal{A}C_q(A_{q,t}^*)\subseteq C_q(A_{q,t}^*).\]
The  space $\mathcal{A}C_q^{(n)}(A_{q,t}^*)$  ($n\in\mathbb{N}$) is the space
of all functions defined on $X$ such that  $f,\, D_qf,\,\ldots,\, D_q^{n-1}f$ are $q$-regular at zero and $D_q^{n-1}f\in \mathcal{A}C_q(A_{q,t}^*)$, c.f.~\cite[Definition 4.3.2]{AMbook}.
Also it is proved in~\cite[Theorem 4.6]{AMbook} that a function $f\in \mathcal{A}C_q^{(n)}(A_{q,t}^*)$ if and only if there exists a function $\phi\in L_q^1(A_{q,t}^*)$
such that
\[f(x)=\sum_{k=0}^{n-1}\frac{D_q^kf(0)}{\Gamma_q(k+1)}x^k+\frac{x^{n-1}}{\Gamma_q(n)}\int_{0}^{x} (qu/x;q)_{n-1}\phi(u)\,d_qu,\;x\in A_{q,t}^*.\]
In particular, $f\in \mathcal{A}C(A_{q,t}^*)$ if and only if $f$ is $q$-regular at zero such that $D_qf\in L_q^1(A_{q,t}^*)$.
It is worth noting that in~\cite{AMbook}, all the definitions and results we have just mentioned are defined and proved for functions defined on the interval $[0,a]$
instead of $A_{q,t}^*$.

\section{Fractional $q$-Calculus}

This section includes the definitions and properties of  the left sided and right sided Riemann--Liouville $q$-fractional operators which we need in our investigations.

The  left   sided  Riemann--Liouville $q$-fractional operator is defined by
\be\label{LS:OP}
I_{q,a^+}^{\alpha}f(x)=\dfrac{x^{\alpha-1}}{\Gamma_q(\alpha)}\int_{a}^{x}(qt/x;q)_{\alpha-1}f(t)\,d_qt.
\ee
 This  definition   is introduced by Agarwal in~\cite{Agr} when  $a=0$ and by Rajkovi\'{c} et.al~\cite{Rajkovic09} for $a\neq 0$.
 We define a right sided Riemann--Liouville $q$-fractional operator  by
\be\label{RS:OP}
I_{q,b^-}^{\alpha}f(x)=\frac{1}{\Gamma_q(\alpha)}\int_{qx}^{b}t^{\alpha-1}(qx/t;q)_{\alpha-1}f(t)\,d_qt.
\ee
One can prove that if $x=bq^m$, $m\in \mathbb{N}_0$, then
\[
\begin{split}
I_{q,b^-}^{\alpha}f(x)&=b^{\alpha}(1-q)^{\alpha}\sum_{j=0}^{m}q^{j\alpha}\frac{(q^{\alpha};q)_{m-j}}{(q;q)_{m-j}}f(bq^j)\\
&=b^{\alpha}(1-q)^{\alpha}\frac{(q^{\alpha};q)_m}{(q;q)_m}\sum_{j=0}^{m}q^{j}\frac{(q^{-m};q)_{j}}{(q^{1-m-\alpha};q)_{j}}f(bq^j),
\end{split}
\]
where we used~\cite[Eq. (I.11)]{GR}.
For example,
\[\begin{split}
I_{q,b^-}^{\alpha}b^{\mu}(\frac{qx}{b};q)_{\mu}&=b^{\alpha+\mu}(1-q)^{\alpha}\dfrac{(q^{\alpha};q)_{m}}{(q;q)_m}(q;q)_{\mu}{}_2\phi_1\left(q^{-m},q^{\mu+1};q^{1-m-\alpha};q,q\right)\\
&=\frac{\Gamma_{q}(\mu+1)}{\Gamma_q(\mu+\alpha+1)}b^{\alpha+\mu}(\frac{qx}{b};q)_{\mu+\alpha}.
\end{split}
\]
\[I_{q,0^+}^{0}f(x)=f(x)\qquad\mbox{and}\qquad I_{q,b^-}^{0}f(x)=f(x);\;x\in A_{q,b}.\]
The left sided Riemann--Liouville $q$-fractional operator satisfies the semigroup property
\[
I_{q,a^+}^\alpha I_{q,a^+}^{\beta}f(x)=I_{q,a^+}^{\alpha+\beta}f(x).
\]
The case $a=0$ is proved in~\cite{Agr} and   the case $a>0$ is proved  in~\cite{Rajkovic09}.

\begin{thm}
The right sided Riemann--Liouville $q$-fractional operator satisfies the semigroup property
\be\label{sgp2}
 I_{q,b^-}^\alpha I_{q,{{b}}^-}^{\beta}f(x)=I_{q,b^-}^{\alpha+\beta}f(x),\;x\in A_{q,b}^*,
\ee
for any function defined on $A_{q,b}$ and for any values of  $\alpha$ and $\beta$.
\end{thm}

\begin{proof}
Let $x\in A_{q,b}^*$. Then
\be\begin{gathered}\label{Eq:7}I_{q,b^-}^{\alpha}I_{q,b^-}^{\beta}f(x)\\=\dfrac{1}{\Gamma_q(\alpha)\Gamma_q(\beta)}\int_{qx}^{b}t^{\alpha-1}(qx/t;q)_{\alpha-1}\int_{qt}^{b}u^{\beta-1}(qt/u;q)_{\beta-1}f(u)\,d_qu\,d_qt\\
=\dfrac{1}{\Gamma_q(\alpha)\Gamma_q(\beta)}\int_{qx}^{b}u^{\beta-1}f(u)\int_{t=qx}^{u}t^{\alpha-1}(qx/t;q)_{\alpha-1}(qt/u;q)_{\beta-1}f(u)\,d_qt\,d_q\,u.
\end{gathered}
\ee
But
\be\label{Eq:8}\int_{qx}^{u}t^{\alpha-1}(qx/t;q)_{\alpha-1}(qt/u;q)_{\beta-1}\,d_qt=u^{\alpha}\beta_q(\alpha,\beta)(qx/u;q)_{\alpha+\beta-1},\ee
then substituting from \eqref{Eq:8} into \eqref{Eq:7} yields~\eqref{sgp2}.
\end{proof}

\begin{defn}
Let $\alpha>0$ and  $\ulcorner\alpha\urcorner=m$.  The left and right side  Riemann--Liouville fractional $q$-derivatives of order $\alpha$ are defined by
\[
D_{q,a^+}^\alpha f(x):=D_q^mI_{q,a^+}^{m-\alpha}f(x),\; D_{q,b^-}^\alpha f(x):=\left(\frac{-1}{q}\right)^mD_{q^{-1}}^mI_{q,b^-}^{m-\alpha}f(x),
\]
the left and right sided Caputo fractional $q$-derivatives of order $\alpha$ are defined by
\[
\CP^{\alpha}f(x):=I_{q,a^+}^{m-\alpha}D_q^mf(x),\quad{}^cD_{q,b^-}^\alpha:=\left(\frac{-1}{q}\right)^mI_{q,b^-}^{m-\alpha}D_{q^{-1}}^mf(x).
\]
\end{defn}
From now on, we shall consider left sided Riemann--Liouville and Caputo fractional $q$-derivatives when the lower point $a=0$ and right sided Riemann--Liouville and Caputo
fractional $q$-derivatives when $b=a$.
According to~\cite[pp. 124,\,148]{AMbook}, $D_{q,0^+}^{\alpha}f(x)$ exists if
\[f\in L_q^1(A_{q,a}^*)\;\;\mbox{such that}\;\;I_{q,0^+}^{m-\alpha}f\in \mathcal{A}C_q^{(m)}(A_{q,a}^*),\]
and $\CP^{\alpha}f$ exists if
\[f \in \mathcal{A}C_q^{(m)}(A_{q,a}^*).\]

\begin{lem}

\begin{enumerate}Let $\alpha\in(0,1)$.

\item[(i)]If $f\in L_q^1(A_{q,a}^*)$  such that $I_{q,0^+}^{\alpha}f\in \mathcal{A}C_q(A_{q,a}^*)$ then

\be\label{CR3}\CZ^{\alpha}I_{q,0^+}^{\alpha}f(x)=f(x)-\dfrac{I_{q,0^+}^{\alpha}f(0)}{\Gamma_q(1-\alpha)}x^{-\alpha}.
\ee
Moreover, if $f$ is bounded on $A_{q,a}^*$ then
\be\label{CR3-}\CZ^{\alpha}I_{q,0^+}^{\alpha}f(x)=f(x).\ee
\item[(ii)] For any function $f$ defined on $A_{q,a}^*$
\be\label{CR3-1}
 \CM^{\alpha}I_{q,a^-}^{\alpha}f(x)=f(x)-\frac{a^{-\alpha}}{\Gamma_q(1-\alpha)}(qx/a;q)_{-\alpha}\left(I_{q,a^-}^{\alpha}f\right)(\frac{a}{q}).\ee

\item[(iii)] If $f\in L_q^1(A_{q,a})$ then
\be\label{CR1}
D_{q,0^+}^{\alpha}I_{q,0^+}^{\alpha}f(x)=f(x).
\ee

\item[(iv)] For any function $f$ defined on $A_{q,a}^*$
\be\label{CR2}
D_{q,a^-}^{\alpha}I_{q,a^-}^{\alpha}f(x)=f(x).
\ee
\item[(v)] If $f\in\mathcal{A}C_q(A_{q,a}^*))$  then
\be\label{CR4}
I_{q,0^+}^{\alpha}\CZ^{\alpha}f(x)=f(x)-f(0).
\ee
\item[(vi)] If $f$ is a function defined on $A_{q,a}^*$ then
\be\label{CR5}
I_{q,a^-}^{\alpha}D_{q,a^-}^{\alpha}f(x)=f(x)-\frac{a^{\alpha-1}}{\Gamma_q(\alpha)}(qx/a;q)_{\alpha-1}\left(I_{q,a^-}^{1-\alpha}f\right)(\frac{a}{q}).
\ee
\item[(v)] If $f$ is defined on $[0,a]$ such that $D_qf$ is continuous on $[0,a]$ then
\be\label{CR6}
\CZ^{\alpha}f(x)=D_{q,0^+}^{\alpha}\left[f(x)-f(0)\right].
\ee
\end{enumerate}
\end{lem}

\begin{proof}
The proof of \eqref{CR3} is a special case of \cite[Eq.~(5.7)]{AMbook} but note that there is a misprint in the formula (5.7), the summation should start from $i=1$.
If $f$ is bounded on $A_{q,a}$, then $I_{q,0^+}^{1-\alpha}f(0)=0$, and \eqref{CR3-} follows at once from \eqref{CR3}.
Now we prove \eqref{CR3-1}.
\[\CM^{\alpha}I_{q,a^-}^{\alpha}f(x)=\frac{1}{\Gamma_q(1-\alpha)}\int_{qx}^{a}t^{-\alpha}(qx/t;q)_{-\alpha}D_q(I_{q,a^-}^{\alpha}f)(\frac{t}{q})\,d_qt,\]
where we used $-\frac{1}{q}D_{q^{-1}}f(x)=D_{q,x}f(\frac{x}{q})$. Then  applying the $q$-integration by parts formula~\eqref{qIR} and using
\[D_{q,t}t^{\beta}(qx/t;q)_{\beta}=-[\beta]t^{\beta-1}(qx/t;q)_{\beta-1},\;\beta\in\mathbb{R},\;[\beta]:=\frac{1-q^{\beta}}{1-q},\]
we obtain
\[\CM^{\alpha}I_{q,a^-}^{\alpha}f(x)=\frac{a^{-\alpha}}{\Gamma_q(1-\alpha)}(qx/a;q)_{-\alpha}\left(I_{q,a^-}^{1-\alpha}f\right)(\frac{a}{q})-I_{q,a^-}^{-\alpha}I_{q,a^-}^{\alpha}f(x).\]
Hence, the result follows from the semigroup property~\eqref{sgp2}.
\eqref{CR1} is proved in~\cite[Eq.~(4.66)]{AMbook}. The proof of \eqref{CR2} follows from the fact that
\[D_{q,a^-}^{\alpha}I_{q,a^-}^{\alpha}f(x)=-\frac{1}{q}I_{q,a^-}^{1-\alpha}I_{q,a^-}^{\alpha}f(x)=-\frac{1}{q}D_{q^{-1}}I_{q,a^-}f(x)=f(x),\]
where we used the semigroup property~\eqref{sgp2}.
The proof of \eqref{CR4} is a special case of~\cite[Eq.~(5.6)]{AMbook}. The proof of \eqref{CR5} is similar to the proof of~\eqref{CR2}and is omitted. Finally,
the proof of \eqref{CR6} is a special case of \cite[Eq.~(5.8)]{AMbook}.
\end{proof}
Set $X=A_{q,a}$ or $A_{q,a}^*$. Then
\[C(X)\subseteq L_q^2(X)\subseteq L_q^1(X).\]
Moreover, if $f\in C(X)$ then
\[\norm{f}_1\leq \sqrt{a}\norm{f}_2\leq a \norm{f}.\]
We have also the following inequalities:
\begin{enumerate}
\item  If $f\in C(A_{q,a}^*)$ then $I_{q,0^+}^{\alpha}f\in C(A_{q,a}^*)$ and
 \be\label{ineq1}   \norm{I_{q,0^+}^{\alpha}f}\leq \frac{a^{\alpha}}{\Gamma_q(\alpha+1)}\norm{f}.\ee
\item If $f\in L_q^1(X)$ then $I_{q,0^+}^{\alpha}f\in L_q^1(X)$ and
\be\label{ineq3}\norm{I_{q,0^+}^{\alpha}f}_1\leq M_{\alpha,1}\norm{f}_1,\qquad M_{\alpha,1}:=\dfrac{a^{\alpha}(1-q)^{\alpha}}{(1-q^{\alpha})(q;q)_{\infty}}.
\ee
\item If $f\in L_q^2(X)$ then $I_{q,0^+}^{\alpha}f\in L_q^2(X)$ and
\be\label{ineq2}
\norm{I_{q,0^+}^{\alpha}f}_2\leq\,M_{\alpha,2}\norm{f}_2,\ee
\[
M_{\alpha,2}:=\frac{a^{\alpha}}{\Gamma_q(\alpha)}\sqrt{\frac{(1-q)}{(1-q^{2\alpha})}}\left(\int_{0}^{1}(q\xi;q)^2_{\alpha-1}\,d_q\xi\right)^{1/2}.\]

\item If $\alpha>\frac{1}{2}$ and  $f\in L_q^2(X)$ then $I_{q,0^+}^{\alpha}f\in C(X)$ and
\be\label{ineq0}
\norm{I_{q,0^+}^{\alpha}f}\leq \widetilde{M}_{\alpha}\norm{f},\;\widetilde{M}_{\alpha}:= \frac{a^{\alpha-\frac{1}{2}}}{\Gamma_q(\alpha)}\left(\int_{0}^{1}(q\xi;q)^2_{\alpha-1}\,d_q\xi\right)^{1/2}.
\ee
\item Since $\norm{f}_2\leq \sqrt{a}\norm{f}$,  we conclude that if $f\in C(X)$ then $I_{q,0^+}^{\alpha}f\in L_q^2(X)$ and
\be\label{ineq3-2}
\norm{I_{q,0^+}^{\alpha}f}_2\leq K_{\alpha}\norm{f},\quad K_{\alpha}:=\sqrt{a}M_{\alpha,2}.
\ee
\item If $f\in C(A_{q,a}^*)$ then $I_{q,a^-}^{\alpha}f\in C(A_{q,a}^*)$ and
\[ \norm{I_{q,a^-}^{\alpha}f}\leq c_{\alpha,0} \norm{f},\quad c_{\alpha,0}:=\frac{a^{\alpha}(1-q)^{\alpha}}{(1-q^{\alpha})(q;q)_{\infty}}.\]
\item If $f\in L_q^1(X)$ then $I_{q,a^-}^{\alpha}f\in L_q^1(X)$ and
\[\norm{I_{q,a^-}^{\alpha}f}_1\leq \left\{\begin{array}{cc}\dfrac{(1-q)^{\alpha}a^{\alpha}}{(1-q^{\alpha})(q;q)_{\infty}}\norm{f}_1,& \mbox{if}\,\alpha<1,\\&\\
\dfrac{(1-q)^{\alpha-1}a^{\alpha-1}}{(q;q)_{\infty}}\norm{f}_1,& \mbox{if}\,\alpha \geq 1.\end{array}\right.\]
\item If $\alpha \neq \frac{1}{2}$  and $f\in L_q^2(X)$ then $I_{q,a^-}^{\alpha}f\in L_q^1(X)$ and
\[\norm{I_{q,a^-}^{\alpha}f}_2\leq \left\{\begin{array}{cc}\dfrac{(1-q)^{\alpha-\frac{1}{2}}a^{\alpha}}{\sqrt{1-q^{2\alpha-1}}(q;q)_{\infty}}\norm{f}_2,& \mbox{if}\,\alpha<\frac{1}{2},\\&\\
\dfrac{(1-q)^{\alpha}a^{\alpha}}{(q;q)_{\infty}\sqrt{(1-q^{2\alpha-1})(1-q^{2\alpha})}}\norm{f}_2,& \mbox{if}\,\alpha > \frac{1}{2}.\end{array}\right.\]
\end{enumerate}

\begin{lem}\label{Lem:1}Let $\alpha>0$. If
\begin{enumerate}
\item[(a)]   $f\in L_q^1(X)$ and $g$ is a bounded function on $A_{q,a}$,

or
\item[(b)] $\alpha\neq \frac{1}{2}$ and $f,\,g$ are $L_q^2(X)$ functions
\end{enumerate}
 then
\be\label{I2}
\int_{0}^{a}g(x)I_{q,0^{+}}^{\alpha}f(x)\,d_qx
=\int_{0}^{a}f(x)I_{q,{a^-}}^{\alpha}g(x)\,d_qx.\ee

\end{lem}
\begin{proof}
The condition (a) or (b) of the present lemma assures the convergence of the $q$-integrals in \eqref{I2}.
Since
\[
\int_{0}^{a}g(x)I_{q,0+}^{\alpha}f(x)\,d_qx=\frac{1}{\Gamma_q(\alpha)}\int_{0}^{a}g(x) x^{\alpha-1}\int_{0}^{x}(qt/x;q)_{\alpha-1}f(t)\,d_qt\,d_qx,\]
from the conditions on the functions $f$ and $g$, the double $q$-integral is absolutely convergent, therefore we can interchange the order of the $q$-integrations to obtain
\[
\begin{gathered}
\int_{0}^{a}g(x)I_{a+}^{\alpha}f(x)\,d_qx=\int_{0}^{a}f(t)\frac{1}{\Gamma_q(\alpha)}\int_{qt}^{a}x^{\alpha-1}(qt/x;q)_{\alpha-1}g(x)\,d_qx \,d_qt\\
=\int_{0}^{a}f(t)I_{q,a^-}^{\alpha}g(t)\,d_qt.
\end{gathered}
\]
\end{proof}
\begin{lem} Let $\alpha\in(0,1)$.
\begin{itemize}
 \item[(a)] If $g\in L_q^1(A_{q,a}^*)$ such that $I_{q}^{1-\alpha}g\in\mathcal{A}C_q(A_{q,a}^*)$, and  $D_q^i f\in C(A_{q,a}^*)$ ($i=0,1$) then
\be\label{CP}
\int_{0}^{a}f(x)D_{q,0^+}^{\alpha}g(x)\,d_qx=-f(\frac{x}{q})I_{q,0^+}^{1-\alpha}g(x)\Big|_{x=0}^a+\int_{0}^{a}g(x) \CM^{\alpha}f(x)\,d_qx.
\ee
\item[(b)] If $f\in\mathcal{A}C_q(A_{q,a}^*)$, and $g$ is a bounded function on $A_{q,a}^*$ such that $D_{q,a^-}^{\alpha}g\in L_q^1(A_{q,a}^*)$ then
\be\label{CP3}
\int_{0}^{a}g(x){}^cD_{q,0^+}^{\alpha}f(x)\,d_qx=\left(I_{q,a^-}^{1-\alpha}g\right)(\frac{x}{q})f(x)\Big|_{x=0}^a+
\int_{0}^{a}f(x) D_{q,a^-}^\alpha g(x)\,d_qx.
\ee
\end{itemize}
\end{lem}
\begin{proof}
The conditions on the functions $f$ and $g$ guarantee the convergence of the $q$-integrals  in \eqref{CP} and \eqref{CP3},
and their proofs  follow from Lemma~\ref{Lem:1} and the $q$-integration by parts rule~\eqref{qIR}.

\end{proof}
\section{Regular Fractional $q$-Sturm Liouville problems}
\begin{defn}
Let $\alpha\in (0,1)$. With the notation
\[
\mathcal{L}_{q,\alpha}y:=D_{q,a^-}^{\alpha}p(x)\CZ^{\alpha}y(x) +r(x)y(x),
\]
consider the fractional $q$-Sturm--Liouville equation
\be\label{EVP1}
\mathcal{L}_{q,\alpha}y(x)-\lambda w_{\alpha}(x)y(x)=0, \quad x\in A_{q,a}^*,
\ee
where $p(x)\neq 0$ and $w_{\alpha}>0$ for all $x\in A_{q,a}^*$, $p,r,\,w_{\alpha}$ are real valued functions defined in $A_{q,a}^*$
and the associated  boundary conditions are
\be\label{BC1}
c_1 y(0)+c_2 \left[I_{q,a^-}^{1-\alpha}\,p\CZ^{\alpha}y\right](0)=0,
\ee
\be\label{BC2}
d_1 y(a)+d_2\left[ I_{q,a^-}^{1-\alpha}\,p\CZ^{\alpha}y\right](\frac{a}{q})=0,\ee
with $c_1^2+c_2^2\neq 0$ and $d_1^2+d_2^2\neq 0$.

\end{defn}
As in the classical case, the problem of finding the complex numbers $\lambda$'s such that the boundary value problem has a non-trivial solution will be called a regular $q$-fractional Sturm--Liouville problem( regular qFSLP).
Such a value $\lambda$, is called an eigenvalue and the corresponding non-trivial solution, the eigenfunction.

In the following, we assume that $0<\alpha<1$ and consider the  subspace of the vector space $L_q^2(A_{q,a}^*)$ of all $q$-regular at zero functions satisfying the boundary conditions
\eqref{BC1}--\eqref{BC2}.
Hence  $V\subseteq L_q^2(A_{q,a}^*)\cap C(A_{q,a}^*)$ and   $V $associated with the  inner product~\eqref{IP} is a Hilbert space.
\begin{thm}$\mathcal{L}_{q,\alpha}$ is a self-adjoint operator on the Hilbert space $V$.
\end{thm}

\begin{proof}
One can prove that for any functions $f,g\in L_q^2(0,a)\cap C(A_{q,a}^*)$, we have
\be\label{IR}
\left<D_{q,a^-}^{\alpha}f,g\right>=-g(x)\left(I_{q,0^+}^{1-\alpha}f\right)(\frac{x}{q})\Big|_{0}^{a}+\left<f, \CZ^{\alpha}g\right>.
\ee
Therefore, for  $u,v\in V$
\[
\begin{gathered}
\left<\mathcal{L}_{q,\alpha}u,v\right>-\left<u,\mathcal{L}_{q,\alpha}v\right>\\
=\left[u(x)\left(I_{q,a^-}^{1-\alpha}\,p\CZ^{\alpha}u\right)(\frac{x}{q})
-v(x)\left(I_{q,a^-}^{1-\alpha}\,p\CZ^{\alpha}u\right)(\frac{x}{q})\right]\Big|_{x=0}^a+.
\end{gathered}
\]
 This yields the Green's identity
 \be\label{GI}\begin{gathered}
 \int_{0}^{a}\left(u(x)\mathcal{L}_{q,\alpha}v(x)-v(x)\mathcal{L}_{q,\alpha}u(x)\right)\,d_qx=\\\left[u(x)\left(I_{q,a^-}^{1-\alpha}p\CZ^{\alpha}u\right)(\frac{x}{q})
-v(x)\left(I_{q,a^-}^{1-\alpha}p\CZ^{\alpha}u\right)(\frac{x}{q})\right]\Bigg|_{x=0}^a. \end{gathered}\ee
If $u$ and $v$ are in the space $V$, then they  satisfy                                                                                                                                                                                                                                    the boundary condition \eqref{BC1} at $x=0$, then
\[\left(
    \begin{array}{cc}
      u(0) &\left(I_{q,a^-}^{1-\alpha} p\CZ^{\alpha}u\right)(0) \\
     v(0) & \left(I_{q,a^-}^{1-\alpha}p\CZ^{\alpha}v(\cdot)\right)(0) \\
    \end{array}
  \right)\left(
           \begin{array}{c}
             c_1 \\
             c_2 \\
           \end{array}
         \right)=\left(
                   \begin{array}{c}
                     0 \\
                     0 \\
                   \end{array}
                 \right).
\]
Since $c_1^2+c_2^2\neq 0$, we should have
\[  u(0)\left(I_{q,a^-}^{1-\alpha}\,p\CZ^{\alpha}v\right)(0) -v(0)  \left(I_{q,a^-}^{1-\alpha}\,p\CZ^{\alpha}v\right)(0)=0.\]
In the same way we see that if $u$ and $v$ satisfy the conditions \eqref{BC2} at $x = a$, then
\[  u(a)\left(I_{q,a^-}^{1-\alpha}\,p\CZ^{\alpha}v\right)(\frac{a}{q}) -v(a)\left(I_{q,a^-}^{1-\alpha}\, p \CZ^{\alpha}v\right)(\frac{a}{q\underline{\underline{\underline{}}}})=0.\]

Hence
\be\label{SAI}\left<\ML_{q,\alpha} u, v\right>=\left<u,\ML_{q,\alpha}v\right>\ee
and the theorem  follows.
\end{proof}
\begin{thm}
The eigenvalues of the regular  qFSLP \eqref{EVP1}--\eqref{BC2}
 are real.\end{thm}

 \begin{proof}
 Assume that $\lambda$ is an eigenvalue associated with an eigenfunction $y$. Then we have
 \[\mathcal{L}_ {q,\alpha}y(x)=\lambda w_{\alpha}(x)y(x),\quad \mathcal{L}_{q,\alpha}\overline{y(x)}=\overline{\lambda} w_{\alpha}(x)y(x). \]
 Therefore from Green's identity \eqref{GI}
 \[
 \begin{gathered}
 (\overline{\lambda}-\lambda)\int_{0}^{a}w_{\alpha}(x)|y(x)|^2\,d_qx=\int_{0}^{a}\left(y(x)\ML_{q,\alpha}\overline{y(x)}-\overline{y(x)}\ML_{q,\alpha}y(x)\right)\,d_qx=0.
 \end{gathered}
 \]
 Since $y$ is a non trivial solution and $w_{\alpha}>0$, we obtain $\lambda=\overline{\lambda}$.
 \end{proof}

 \begin{lem}
 If $u$ and $v$ are eigenfunctions of the regular qFSLP \eqref{EVP1}--\eqref{BC2} associated with different eigenvalues $\lambda\,$  and $\mu$, then
 $u$ and $v$ are orthogonal on the weighted space $L_q^2(A_{q,a}^*,w_{\alpha})$.

 \end{lem}
 \begin{proof}
 Since $\mathcal{L}_{q,\alpha}$ is self adjoint, then substituting with $\mathcal{L}_{q,\alpha}u=\lambda\,w_{\alpha} u$
 and $\mathcal{L}_{q,\alpha}v=\mu\, w_{\alpha} v$ in the identity  \eqref{SAI} taking into consideration that the eigenvalues are real,
 we obtain
 \[
 (\lambda-\mu)\int_{0}^{a}u(x)v(x)w_{\alpha}(x)\,d_qx=0.
 \]
 Since $\lambda \neq \mu$, then $\int_{0}^{a}u(x)v(x)w_{\alpha}(x)\,d_qx=0$ and the lemma follows.
 \end{proof}

 \bigskip

 In the following we use the fixed point theorem to show that for the  regular qFSLP~\eqref{EVP1}--\eqref{BC2} if the  eigenvalues  satisfying certain condition,
 then the eigenfunctions are unique up to a multiplying constant on the space $C(A_{q,a}^*)$ for any $a>0$. We also prove that under a  certain constrain on the domain of solutions,  for any eigenvalue $\lambda$, the eigenfunction is unique up to a constant  multiplying factor.
 Since $I_{q,a^-}^{\alpha}(1)=\frac{a^{\alpha}}{\Gamma_q(\alpha+1)}(qx/a;q)_{\alpha}$,
and
\[
\begin{split}
\left(I_{q,0^+}^{\alpha}I_{q,a^-}^{\alpha}(1)\right)(x)&=I_{q,0^+}^{\alpha}\dfrac{a^{\alpha}(qx/a;q)_{\alpha}}{\Gamma_q(\alpha+1)}\\
&=\frac{a^{\alpha}x^{\alpha}}{\Gamma_q^{2}(\alpha+1)}{}_2\phi_1\left(q^{-\alpha},q;q^{\alpha+1},q,\frac{xq^{\alpha+1}}{a}\right)=:\phi(x),
\end{split}
\]
for $|\frac{xq^{\alpha+1}}{a}|<1$.
 The general solution of the equation
 \[
 D_{q,a^-}^{\alpha}p(x)\CZ^{\alpha}\phi_0(x)=0
 \]
 takes the form
 \be\label{psi}
 \phi_0(x)=\xi_1+\xi_2I_{q,0^+}^{\alpha}\dfrac{a^{\alpha-1}(qx/a;q)_{\alpha-1}}{\Gamma_q(\alpha)p(x)}=:\xi_1+\xi_2\psi_{\alpha,a}(x).
 \ee

 \begin{lem}
Let $\alpha\in (0,1)$ and
\be\begin{split}\label{delta}
Y_y(x)&:=r(x)y(x)-\lambda w_{\alpha}(x)y(x),\\
\Delta&:=c_1d_2-c_2d_1+c_1d_1\psi_{\alpha,a}(a).
\end{split}
\ee
Assume that $\Delta\neq 0$, then on the space $C(A_{q,a}^{*})$, the regular $q$-Sturm--Liouville problem \eqref{EVP1}--\eqref{BC2} is equivalent to the $q$-integral equation
\be\label{Evp2}
\begin{split}
y(x)&=-I_{q,0^+}^{\alpha}\left(\frac{1}{p(\cdot)}I_{q,a^-}^{\alpha}Y_y\right)(x)+A(x)\int_{0}^{a}Y_y(x)\,d_qx\\
&+B(x) \left(I_{q,0^+}^{\alpha}\frac{1}{p(\cdot)}I_{q,a^-}^{\alpha}Y_y\right)(x)\Big|_{x=a},
\end{split}
\ee
where the coefficients $A(x)$ and $B(x)$ are
\begin{eqnarray*}
A(x)&:=&\frac{c_2}{\Delta}\left[d_2+d_1\left(\psi_{\alpha,a}(a)-\psi_{\alpha,a}(x)\right)\right]\\
B(x)&:=&\frac{d_1}{\Delta}\left[c_1\psi_{\alpha,a}(x)-c_2\right]
\end{eqnarray*}
and the function $\psi_{\alpha,a}$ is defined in \eqref{psi}.
 \end{lem}

 \begin{proof}
 Using \eqref{delta},  we can rewrite \eqref{EVP1} as follows:
\[
D_{q,a^-}^{\alpha}p(x)\CZ^{\alpha}\left[y(\cdot)+I_{q,0^+}^{\alpha}\frac{1}{p(\cdot)}I_{q,a^-}^{\alpha}Y_y\right](x)=0.
\]
Thus,
\[
y(x)+I_{q,0^+}^{\alpha}\left[\frac{1}{p(\cdot)}I_{q,a^-}^{\alpha}Y_y\right](x)=\xi_1+\xi_2\psi_{\alpha,a}(x).
\]
From the boundary conditions \eqref{BC1}--\eqref{BC2},
we have
\begin{eqnarray*}
\xi_1&=&y(0)\\
\xi_2&=&\int_{0}^{a}Y_y(x)\,d_qx+\left(I_{q,a^-}^{1-\alpha}p\,\CZ^{\alpha}y\right)(0)\\
\xi_1+\xi_2 I_{q,0^+}^{\alpha}\dfrac{a^{\alpha-1}(qx/a;q)_{\alpha-1}}{\Gamma_q(\alpha) p(x)}\Big|_{x=a}&=&y(a)+\left(I_{q,0^+}^{\alpha}\frac{1}{p}I_{q,a^-}^{\alpha}Y_y\right)(a)\\
\xi_2&=&\left(I_{q,a^-}^{1-\alpha}p\,\CZ^{\alpha}Y_y\right)(\frac{a}{q}).
\end{eqnarray*}
This leads to the system of equations
\begin{eqnarray*}
c_1\xi_1+c_2\xi_2&=&c_2X\\
d_1\xi_1+(d_2+d_1\psi_{\alpha,a}(a))\xi_2&=&d_1 Z,
\end{eqnarray*}
where
\[
X:=\int_{0}^{a}Y(y)(x)\,d_qx,\quad\mbox{and}\quad Z:=\left(I_{q,0^+}^{\alpha}\frac{1}{p}I_{q,a^-}^{\alpha}Y_y\right)(a).
\]
Since $\Delta\neq 0$, the solution for coefficients $\xi_j$, $j=1,2$, is unique:

\[
\xi_1=\frac{c_2}{\Delta}\left(X\left(d_2+d_1\psi_{\alpha,a}(a)\right)-d_1Z\right),\qquad\xi_2=\frac{d_1}{\Delta}\left(c_1Z-c_2X\right).
\]
\end{proof}
Let us introduce the notation
\begin{align}
A&:=\norm{A(x)}&m_p&:=\inf_{x\in\,A_{q,a}}|p(x)|\\
B&:=\norm{B(x)}&M_{\phi}&:=\norm{\phi(x)}.
\end{align}
\begin{thm}\label{EUT:1} Let $0<\alpha<1$. Assume that $\Delta\neq 0$. Then unique $q$-regular at zero function $y_{\lambda}$ for the regular $q$FSLP \eqref{EVP1} with
the boundary conditions \eqref{BC1}--\eqref{BC2} corresponding to each eigenvalue obeying
\be\label{EVC} \norm{r-\lambda w_{\alpha}}< \frac{m_p}{M_{\phi}+B\phi(a)+A am_p}\ee
exists and such eigenvalue is simple.
\end{thm}

\begin{proof}
One can verify that \eqref{EVP1} can be interpreted as a fixed point for the mapping
$T:C(A_{q,a}^*)\to  C(A_{q,a}^*)$ defined by
\[
\begin{split}
Tf(x)&=-I_{q,0^+}^{\alpha}\left[\frac{1}{p(\cdot)}I_{q,a^-}^{\alpha}Y_f\right](x)+A(x)\int_{0}^{a}Y_f(x)\,d_qx\\
&+B(x) I_{q,0^+}^{\alpha}\left[\frac{1}{p(\cdot)}I_{q,a^-}^{\alpha}Y_f(\cdot)\right](x)\Big|_{x=a},
\end{split}
\]
using the estimate
\[\norm{Y_g-Y_h}\leq \norm{g-h}\norm{r-\lambda w_{\alpha}},
\]
then
\begin{eqnarray}\label{Ineq4}
\norm{Tg-Th}&\leq& \norm{I_{q,0^+}^{\alpha}\frac{1}{p(x)}I_{q,a}^{\alpha}(Y_g-Y_h)}+\norm{A(x)}\norm{\int_{0}^{a}Y_g(x)-Y_h(x)\,d_qx}\non\\&&+
\norm{B(x)}\left|I_{q,0^+}^{\alpha}\left[\frac{1}{p(\cdot)}I_{q,a^-}^{\alpha}\left(Y_g-Y_h\right)(\cdot)\right](x)\Big|_{x=a}\right|\non\\
&\leq& \norm{g-h}\norm{r-\lambda w_{\alpha}}\left(\dfrac{\norm{\phi}}{m_p}+Aa+\frac{B\phi(a)}{m_p}\right)\\
&=&\norm{g-h}L,\non
\end{eqnarray}
where $L=\norm{r-\lambda w_{\alpha}}\left(\dfrac{M_\phi}{m_p}+Aa+\frac{B\phi(a)}{m_p}\right)$. Using the assumption of the theorem, we conclude that there is a unique fixed point denoted by $y_{\lambda}\in C(A_{q,a}^*)$ satisfies \eqref{EVP1} or equivalently \eqref{Evp2} and the boundary conditions \eqref{BC1}--\eqref{BC2}. Therefore, such eigenvalue is simple.

\end{proof}

\begin{thm}\label{EUT:1-0} Let $0<\alpha<1$ and $k_i$ $(i=0,1)$ be  real numbers. Assume that the functions $p,\,r$, and $w_{\alpha}$ are $C(A_{q,a}^*)$ functions such that $\inf_{x\in A_{q,a}}\,p(x)>0$. Then, there exists $m_0\in\mathbb{N}_0$ such that
 the regular $q$FSLP \eqref{EVP1} with
the initial conditions
\be \label{con:0}y(0)=k_0,\quad   \left(I_{q,a^-}^{1-\alpha}p{}^cD_{q}^{\alpha}y\right)(0)=k_1\ee
has  a unique solution in $C(A_{q,aq^{m_0}}^*)$.
\end{thm}

\begin{proof}
Let  $y_1$ and $y_2$ be two solutions of \eqref{EVP1} satisfying \eqref{con:0}. Set $z:=y_1-y_2$, hence $z$ is a solution of \eqref{EVP1} with the conditions
\be\label{con:2}z(0)=0,\quad   \left(I_{q,a^-}^{1-\alpha}p{}^cD_{q}^{\alpha}z\right)(0)=0.\ee
In this case, simple manipulations show that  \eqref{EVP1} can be interpreted as a fixed point for the mapping
$T:C(A_{q,a}^*)\to  C(A_{q,a}^*)$ defined by
\be\label{Tdef-0}
\begin{split}
Tf(x)&=-I_{q,0^+}^{\alpha}\left[\frac{1}{p(\cdot)}I_{q,a^-}^{\alpha}Y_f\right](x)\\
&+\psi_{\alpha,a}(x)\int_{0}^{a}Y_f(x)\,d_qx,
\end{split}
\ee
using the estimate
\be\label{E2}
\norm{Y_g-Y_h}\leq \norm{g-h}\norm{r-\lambda w_{\alpha}},
\ee
and  the inequality
\[|\psi_{\alpha,a}(x)|\leq  \frac{C}{m_p}a^{\alpha}x^{\alpha},\; C:=\dfrac{q^{-\alpha}[\alpha]}{(q^{\alpha+1};q)_{\infty}\Gamma_q^2(\alpha+1)},\]
for all $x\in A_{q,a}$. Therefore if $x\in A_{q,aq^m}$, $m\in \mathbb{N}$, then
\[\norm{T(g-h)}\leq C \dfrac{\norm{r-\lambda w_{\alpha}}}{m_p}a^{2\alpha} q^{m\alpha}\norm{g-h}.\]
 We can choose $m_0\in \mathbb{N}$ such that
\[C \dfrac{\norm{r-\lambda w_{\alpha}}}{m_p}a^{2\alpha} q^{m_0\alpha}<1.\]
Thus $T:C(A_{q,aq^{m_0}}^*)\to  C(A_{q,aq^{m_0}}^*)$ is a contraction mapping. Hence $z$ is the unique  fixed point of \eqref{Tdef-0}. Therefore,  $z\equiv0$ on $A_{q,aq^{m_0}}$.I.e.
$y_1=y_2$ on $A_{q,aq^{m_0}}$.

\end{proof}

Another version of Theorem~\ref{EUT:1} holds if we release the conditions on  the functions $r$ and $w_{\alpha}$ to be only $L_q^2(A_{q,a}^*)$ functions.

\begin{thm}\label{EUT:1-1} Let $\frac{1}{2}<\alpha<1$. Assume that the functions $r$, and $w_{\alpha}$ are $ L_q^2(A_{q,a}^*)$ functions and $p$ is a functions satisfying
 $\inf_{x\in A_{q,a}}\,p(x)>0$.
  If $\Delta\neq 0$, then unique $q$-regular at zero function $y_{\lambda}$ for the regular $q$FSLP \eqref{EVP1} with
the boundary conditions \eqref{BC1}--\eqref{BC2} corresponding to each eigenvalue obeying
\be\label{EVC2} \norm{r-\lambda w_{\alpha}}_2< \frac{m_p}{(1+B)a^{2\alpha-\frac{1}{2}}c_{\alpha}+A \sqrt{a}m_p},\;c_{\alpha}:=\frac{(1-q)^{\alpha-\frac{1}{2}}}{(q;q)_{\infty}\sqrt{1-q^{2\alpha-1}}}\ee
if $\frac{1}{2}<\alpha<1$,
and obeying
\be\label{EVC2-2} \norm{r-\lambda w_{\alpha}}_2< \frac{m_p}{(1+B)a^{2\alpha-\frac{1}{2}}\gamma_{\alpha}+A \sqrt{a}m_p},
\ee
where
\[
\gamma_{\alpha}:=\frac{\Gamma_q(\alpha+\frac{1}{2})}{(q^{\alpha};q)_{\infty}\Gamma_q(2\alpha+\frac{1}{2})}\sqrt{\frac{1-q}{1-q^{1-2\alpha}}}\]
if $\frac{1}{4}<\alpha<\frac{1}{2}$
exists and such eigenvalue is simple.
\end{thm}

\begin{proof}
Similar to the proof of Theorem~\ref{EUT:1},   \eqref{EVP1} can be interpreted as a fixed point for the mapping
$T:C(A_{q,a}^*)\to  C(A_{q,a}^*)$ defined by
\be\label{Tdef}
\begin{split}
Tf(x)&=-I_{q,0^+}^{\alpha}\left[\frac{1}{p(\cdot)}I_{q,a^-}^{\alpha}Y_f\right](x)+A(x)\int_{0}^{a}Y_f(x)\,d_qx\\
&+B(x) I_{q,0^+}^{\alpha}\left[\frac{1}{p(\cdot)}I_{q,a^-}^{\alpha}Y_f(\cdot)\right](a).
\end{split}
\ee

\noindent{\bf Case 1: $\frac{1}{2}<\alpha<1$ }
Using the estimate
\be\label{E2-2}
\begin{gathered}
\norm{I_{q,a^-}^{\alpha}(Y_g-Y_h)}\\\leq \norm{g-h}\norm{r-\lambda w_{\alpha}}_2\frac{1}{\Gamma_q(\alpha)}\left(\int_{qx}^{a}t^{2\alpha-2}(qx/t;q)_{\alpha-1}^2\,d_qt\right)^{1/2},
\end{gathered}
\ee
then if $\alpha>\frac{1}{2}$, we obtain
\[\begin{split}\int_{qx}^{a}t^{2\alpha-2}(qx/t;q)_{\alpha-1}^2\,d_qt&\leq \frac{1}{(q^{\alpha};q)_{\infty}^2}\int_{qx}^{a}t^{2\alpha-2}\,d_qt\\
&=a^{2\alpha-1}\frac{(1-q)}{(q^{\alpha};q)_{\infty}^2}\dfrac{(1-q^{(2\alpha-1)(m+1)})}{1-q^{2\alpha-1}}\leq \frac{a^{2\alpha-1}}{(q^{\alpha};q)_{\infty}^2}\frac{1-q}{1-q^{2\alpha-1}}.
\end{split}
\]
Consequently,
\[\begin{gathered}\left|I_{q,0^+}^{\alpha}\left(\frac{1}{p}I_{q,a^-}^{\alpha}(Y_g-Y_h)\right)(x)\right|\\
\leq\,\norm{g-h}\norm{r-\lambda w_{\alpha}}_2\frac{a^{\alpha-\frac{1}{2}}}{m_p(q^{\alpha};q)_\infty}\sqrt{\frac{1-q}{1-q^{2\alpha-1}}}\left(I_{q,0^+}^{\alpha}1\right)\\
\leq \norm{g-h}\norm{r-\lambda w_{\alpha}}_2\frac{a^{2\alpha-\frac{1}{2}}}{m_p(q;q)_\infty}\frac{(1-q)^{\alpha-\frac{1}{2}}}{\sqrt{1-q^{2\alpha-1}}}.
\end{gathered}\]
A simple manipulation gives
\be \label{ineq5}\left|\int_{0}^{a}\left(Y_g-Y_h\right)(x)\,d_qx\right|\leq \norm{g-h}\norm{r-\lambda w_{\alpha}}_2\sqrt{a}.\ee
Therefore,
\be\label{Ineq4-2}\begin{split}
&\norm{Tg-Th}\\
&\leq  \norm{g-h}\norm{r-\lambda w_{\alpha}}_2\left(\frac{(1+B)}{m_p(q;q)_{\infty}\sqrt{1-q^{2\alpha-1}}}a^{2\alpha-\frac{1}{2}}(1-q)^{\alpha-\frac{1}{2}}+A\sqrt{a}\right)\\
&=L_1\norm{g-h},
\end{split}
\ee
where $L_1=\norm{r-\lambda w_{\alpha}}_2\left(\frac{(1+B)}{m_p(q;q)_{\infty}\sqrt{1-q^{2\alpha-1}}}a^{2\alpha-\frac{1}{2}}(1-q)^{\alpha-\frac{1}{2}}+A\sqrt{a}\right)$. Using the assumption of the theorem, we conclude that there is a unique fixed point denoted by $y_{\lambda}\in C(A_{q,a}^*)$ satisfies \eqref{EVP1} or equivalently \eqref{Evp2} and the boundary conditions \eqref{BC1}--\eqref{BC2}. Therefore, such eigenvalue is simple.
\vskip .5 cm

\noindent{\bf Case 2: $\frac{1}{4}<\alpha<\frac{1}{2}$}
In this case, we have
\[\int_{qx}^{a}t^{2\alpha-2}(qx/t;q)_{\alpha-1}^2\,d_qt\leq \frac{x^{2\alpha-1}}{(q^{\alpha};q)_{\infty}^2}\frac{1-q}{1-q^{1-2\alpha}}.\]
Consequently, if we set $\sigma_{\alpha}:=\sqrt{\frac{1-q}{1-q^{1-2\alpha}}}\frac{1}{(q^{\alpha};q)_{\infty}}$ then
\[\begin{gathered}
\left|I_{q,0^+}^{\alpha}\left(\frac{1}{p}I_{q,a^-}^{\alpha}(Y_g-Y_h)\right)(x)\right|\leq \sigma_{\alpha}\norm{g-h}\norm{r-\lambda w_{\alpha}}_2\left(I_{q,0^+}^{\alpha}t^{\alpha-\frac{1}{2}}\right)(x)\\
\leq \sigma_{\alpha}\norm{g-h}\norm{r-\lambda w_{\alpha}}_2\frac{\Gamma_q(\alpha+\frac{1}{2})}{\Gamma_q(2\alpha+\frac{1}{2})}x^{2\alpha-\frac{1}{2}}\\
\leq \norm{g-h}\norm{r-\lambda w_{\alpha}}_2\gamma_{\alpha}a^{2\alpha-\frac{1}{2}},
 \end{gathered}
\]
for all $x\in A_{q,a}$. Therefore, using \eqref{ineq5}, we obtain
\begin{eqnarray}\label{Ineq4-3}
\norm{Tg-Th}&\leq& \norm{g-h}\norm{r-\lambda w_{\alpha}}_2\left(\frac{(1+B)}{m_p}\gamma_{\alpha}a^{2\alpha-\frac{1}{2}}+A\sqrt{a}\right)\\
&=&\norm{g-h}L_2,\non
\end{eqnarray}
where $L_2=\norm{r-\lambda w_{\alpha}}_2\left(\frac{(1+B)}{m_p}\gamma_{\alpha}a^{2\alpha-\frac{1}{2}}+A\sqrt{a}\right)$. Using the assumption of the theorem, we conclude that there is a unique fixed point denoted by $y_{\lambda}\in C(A_{q,a}^*)$ satisfies \eqref{EVP1} or equivalently \eqref{Evp2} and the boundary conditions \eqref{BC1}--\eqref{BC2}. Therefore, such eigenvalue is simple.

\end{proof}
\section{{\bf The associated  Wronskian} }
\begin{defn}
Let $y_1,\,y_2$  be two  functions in $\mathcal{A}C_q(A_{q,a}^*)$ and let $0<\alpha<1$. Assume that $p\in C(A_{q,a}^*)$ is a positive function.  The $q,p,\alpha$ Wronskian of $y_1$
and $y_2$  is denoted by $W_{q,p,\alpha}(y_1,y_2)$  and defined by
\[
\W(y_1,y_2)(x)=y_1(x)I_{q,a^-}^{1-\alpha}\left(p{}^cD_{q,0^+}^{\alpha}y_2\right)(x)-y_2(x)I_{q,a^-}^{1-\alpha}\left(p{}^cD_{q,0^+}^{\alpha}y_1\right)(x).
\]
\end{defn}
\begin{thm}
If $y_1$ and $y_2$ are two solutions of \eqref{EVP1}--\eqref{BC2}, then
\[\W(y_1,y_2)(0)=\W(y_1,y_2)(a)\]
\end{thm}
\begin{proof}
Let $y_1$ and $y_2$ be two solutions of \eqref{EVP1}--\eqref{BC2}. Then
\bea\label{eq:1}
D_{q,a^-}^{\alpha}\left(p\CZ^{\alpha}\,y_1\right)(x)+r(x)y_1(x)=\lambda w_{\alpha}y_1(x),\\
\label{eq:2}
D_{q,a^-}^{\alpha}\left(p\CZ^{\alpha}\,y_2\right)(x)+r(x)y_2(x)=\lambda w_{\alpha}y_2(x),
\eea
for all $x\in A_{q,a}^*$.
Multiply \eqref{eq:1} by $y_2$ and \eqref{eq:2} by $y_1$ and subtracting  the two equations. This gives
\be\label{eq:3}
y_1(x)D_{q,a^-}^{\alpha}\left(p\CZ^{\alpha}\,y_2\right)(x)-y_2(x)D_{q,a^-}^{\alpha}\left(p\CZ^{\alpha}\,y_1\right)(x)=0.
\ee
using that $-\frac{1}{q}D_{q^{-1}}f(x)=D_{q,x}f(\frac{x}{q})$, \eqref{eq:3} can be written as
\be\label{eq:4}
y_1(x)D_{q,x}\left(I_{q,a^-}^{1-\alpha}p\CZ^{\alpha}\,y_2\right)(\frac{x}{q})-y_2(x)D_{q,x}\left(I_{q,a^-}^{1-\alpha}p\CZ^{\alpha}\,y_1\right)(\frac{x}{q})=0.
\ee
Hence,
\be\label{eq:5}
D_{q}\W(y_1,y_2)(x)=D_qy_1(x)\left(I_{q,a^-}^{1-\alpha}\,p\CZ^{\alpha}\,y_2\right)(x)-D_qy_2(x)\left(I_{q,a^-}^{1-\alpha}\,p\CZ^{\alpha}\,y_1\right)(x).
\ee
Thus from~\eqref{I2}
\[\begin{gathered}\int_{0}^{a}D_{q}\W(y_1,y_2)(x)\,d_qx\\
=\int_{0}^{a}\left[D_qy_1(x)I_{q,a^-}^{1-\alpha}\left(P\CZ^{\alpha}\,y_2\right)(x)-D_qy_2(x)I_{q,a^-}^{1-\alpha}\left(P\CZ^{\alpha}\,y_1\right)(x)\right]\,d_qx
=0.
\end{gathered}\]
Hence,
\[\W(y_1,y_2)(0)=\W(y_1,y_2)(a).\]
\end{proof}

We  have the following theorems:
\begin{thm}\label{Thm:Wronskian}
Let $y_1$ and $y_2$ be two  functions in $\mathcal{A}C_q(A_{q,a}^*)$.  Then
$y_1$ and $y_2$ are   two  linearly independent solutions of \eqref{EVP1} if and only if $\W(y_1,y_2)(0)\neq 0$.   \end{thm}

\begin{proof}
Let $y_1$ and $y_2$ be two solutions of \eqref{EVP1}  such that $\W(y_1,y_1)(0)\neq 0$. If
\be\label{1} k_1y_1(x)+k_2y_2(x)=0,\quad x\in A_{q,a}^*,\ee
then
\[
\begin{split}
k_1y_1(0)+k_2y_2(0)&=0,\\
k_1\left(I_{q,a^-}^{1-\alpha}p\CZ^{\alpha}\,y_1\right)(0)+k_2\left(I_{q,a^-}^{1-\alpha}p\CZ^{\alpha}\,y_2\right)(0)&=0,
\end{split}
\]
but $\W(y_1,y_2)(0)\neq 0$ implies that $k_1=k_2=0$. I.e. $y_1$ and $y_2$ are linearly independent.
To prove the necessary condition, we suppose on the contrary that $y_1$ and $y_2$ are linearly independent solutions and $\W(y_1,y_2)(0)=0$. Hence there exist constants $r_1$ and $r_2$ not both zeros such that
\[\begin{split}r_1y_1(0)+r_2y_2(0)&=0,\\
r_1\left(I_{q,a^-}^{1-\alpha}p\CZ^{\alpha}\,y_1\right)(0)+r_2\left(I_{q,a^-}^{1-\alpha}p\CZ^{\alpha}\,y_2\right)(0)&=0.
\end{split}\]
Set $y:=r_1y_1+r_2y_2$. Hence, $y$ is a solution  of \eqref{EVP1} satisfying the initial conditions
\be\label{BC7}y(0)=\left(I_{q,a^-}^{1-\alpha}p\CZ^{\alpha}\,y\right)(0).\ee
 According to Theorem~\ref{EUT:1-0}, there exists $m_0\in \mathbb{N}$ such that \eqref{EVP1} with \eqref{BC7} has a unique solution in $C(A_{q,aq^{m_0}}^*)$. Hence, $y\equiv 0$. I.e.
$y_1$ and $y_2$ are linearly dependent which is a contradiction. Hence, we should have $\W(y_1,y_2)(0)=0$.
\end{proof}

\begin{thm}\label{thm:GM}
 The geometric multiplicity of each eigenvalue of the qFSLP \eqref{EVP1}--\eqref{BC2} is 1. I.e. for each eigenvalue,  the associated eigenfunction is unique except for a constant multiplier.
\end{thm}
\begin{proof}
Assume that $y_1$ and $y_2$ are two solutions of \eqref{EVP1}-\eqref{BC2}. Then
\[\begin{split}
c_1y_1(0)+c_2I_{q,a^-}^{1-\alpha}(p\CZ^{\alpha} y_1)(0)&=0,\\
c_1y_2(0)+c_2I_{q,a^-}^{1-\alpha}(p\CZ^{\alpha} y_2)(0)&=0.
\end{split}
\]
Since $c_1$ and $c_2$ are not both zeros, then
\[y_1(0)I_{q,a^-}^{1-\alpha}(p\CZ^{\alpha} y_2)(0)-y_2(0)I_{q,a^-}^{1-\alpha}(p\CZ^{\alpha} y_1)(0)=\W(y_1,y_2)(0)=0.\]
Consequently, from Theorem~\ref{Thm:Wronskian}, $y_1$ and $y_2$ are linearly dependent. I.e. $y_2$ is a constant multiplier of $y_1$.
\end{proof}

\section{ A discrete spectrum of a  qFSLP }
In this section,  we solve  the fractional $q$-Sturm--Liouville problem
\be\label{EX1} D_{q,a^-}^{\mu} (qx;q)_{\beta+\mu}{}^cD_{q,0^+}^{\mu}y(x)=\lambda x^{-\mu}(qx;q)_{\beta} y(x),\; 0<\mu<1,\,x\in A_{q}^*,\ee
under the boundary conditions
\be\label{BC5}y(0)=\left[I_{q,1^-}^{1-\mu}(q^{\beta+1}x;q)_{\mu}{}^cD_{q,0^+}^{\mu} y(x)\right](\frac{1}{q})=0 .\ee
We  show that it  has a discrete  spectrum  $\set{\phi_n,\lambda_n} $ where $\phi_n$  is a   little $q$-Jacobi polynomial and the eigenvalues $\set{\lambda_n}$
have  no finite limit points.  To achieve our goal, we need the following preliminaries.
The little $q$-Jacobi polynomial is defined by, see~\cite[P. 92]{KS}
 \[
 p_n(x;q^{\alpha},q^{\beta}|q)={}_2\phi_1\left(q^{-n},q^{\alpha+\beta+n+1};q^{\alpha+1};q,qx\right),\;\alpha>-1,\,\beta>-1.
 \]
 It satisfies the orthogonality relation
 \be\label{ORJ}
 \int_{0}^{1}w_{\alpha,\beta}p_m(t;q^{\alpha},q^{\beta}|q)p_n(t;q^{\alpha},q^{\beta}|q)\,d_qt=C_n(\alpha,\beta)\,\delta_{n,m},
 \ee
 where
 \[w_{\alpha,\beta}(t):=t^{\alpha}(qt;q)_{\beta},\;\alpha>-1,\,\beta>-1,\]
 and
 \[\begin{gathered}
 C_{n}(\alpha,\beta)\\:=q^{(\alpha+1)n}\dfrac{(1-q)(1-q^{\alpha+\beta+1})}{(1-q^{\alpha+\beta+1+2n})}\dfrac{(q,q^{\alpha+\beta+2};q)_{\infty}}{(q^{\alpha+1},q^{\beta+1};q)_{\infty}}\dfrac{(q,q^{\beta+1};q)_n}{(q^{\alpha+1},q^{\alpha+\beta+1};q)_n}.
  \end{gathered}\]
 \begin{lem}\label{lem:2}
  \[
  I_{q,0^+}^{\mu}\left((\cdot)^{\alpha}p_n(\cdot;q^{\alpha},q^{\beta}|q)\right)(x)=\dfrac{\Gamma_q(\alpha+1)}{\Gamma_q(\mu+\alpha+1)}x^{\alpha+\mu}p_n(x;q^{\alpha+\mu},q^{\beta-\mu}|q).
  \]

 \end{lem}
\begin{proof}
 Using the $q$-beta integral and the series representation of the little $q$-Jacobi polynomial, we can prove
 \be\label{lqj2}
 p_n(x;q^{\alpha+\mu},q^{\beta-\mu}|q)=\dfrac{\Gamma_q(\mu+\alpha+1)}{\Gamma_q(\mu)\Gamma_q(\alpha+1)}\int_{0}^1y^{\alpha}(qy;q)_{\mu-1}p_n(xy;q^{\alpha},q^{\beta}|q)\,d_qy.
 \ee
 Make the substitution
 $t=xy$ on the $q$-integral of the right hand side of \eqref{lqj2}, we obtain

 \be\label{lqj3}
  p_n(x;q^{\alpha+\mu},q^{\beta-\mu}|q)=\dfrac{x^{-\alpha-1}\Gamma_q(\mu+\alpha+1)}{\Gamma_q(\mu)\Gamma_q(\alpha+1)}\int_{0}^{x}t^{\alpha}(qt/x;q)_{\mu-1}p_n(t;q^{\alpha},q^{\beta}|q)\,d_qt.
  \ee

\end{proof}
\begin{cor}\label{cor:1}
\[\begin{gathered}
{}^cD_{q,0^+}^{\mu}\left(x^{\alpha+\mu}p_n(x;q^{\alpha+\mu},q^{\beta-\mu}|q)\right)\\=
\left\{\begin{array}{cc}
\dfrac{\Gamma_q(\mu+\alpha+1)} {\Gamma_q(\alpha+1)} x^{\alpha}p_n(x;q^{\alpha},q^{\beta}|q),&\alpha>-\mu\\&\\
\dfrac{1} {\Gamma_q(1-\mu)} x^{-\mu}\left[p_n(x;q^{\alpha},q^{\beta}|q)-1\right],&\alpha=-\mu.
\end{array}\right.
  \end{gathered}\]
\end{cor}
\begin{proof}
The proof follows form \eqref{CR3} and Lemma \ref{lem:2}.
\end{proof}

  \begin{lem}\label{Lem:3}If $\alpha,\beta$, and $\mu$ are real numbers satisfying
  \[\alpha>-1,\;\beta>-1, \;\beta-1<\mu<\alpha+1\]
  \[
  \begin{gathered}
I_{q,{1}^-}^\mu\left((qt;q)_{\beta}p_m(t; q^\alpha,q^\beta|q  )\right)=\\ \Scale[.9]{q^{m\mu}\dfrac{\Gamma_q(\beta+m+1)\Gamma_q(\alpha-\mu+1+m)\Gamma_q(\alpha+1)}{\Gamma_q(\mu+\beta+m+1)\Gamma_q(\alpha+m+1)\Gamma_q(\alpha-\mu+1)}(qt;q)_{\beta+\mu}p_m(t;q^{\alpha-\mu},q^{\beta+\mu}|q)}.
  \end{gathered}
 \]
  \end{lem}

  \begin{proof}
  In \eqref{lqj3} replace $\alpha$ by $\alpha-\mu$, $\beta$ by $\beta+\mu$ and then substitute in the orthogonality relation \eqref{ORJ}. This gives
  \be\label{lqj4}\begin{gathered}
  C_{n}(\alpha,\beta)\delta_{n,m}= \dfrac{
  \Gamma_q(\alpha+1)}{\Gamma_q(\mu)\Gamma_q(\alpha-\mu+1)}\times\\
  \int_{0}^{1}t^{\mu-1}(qt;q)_{\beta}p_n(t;q^{\alpha},q^\beta|q)\int_{0}^{t}u^{\alpha-\mu}(qu/t;q)_{\mu-1}p_m(u;q^{\alpha-\mu},q^{\beta+\mu}|q)\,d_qu\,d_qt.
  \end{gathered}
  \ee
Changing the order of the $q$-integration gives
  \[\begin{gathered}
  C_{n}(\alpha,\beta)\delta_{n,m}=\\\dfrac{
  \Gamma_q(\alpha+1)}{\Gamma_q(\mu)\Gamma_q(\alpha-\mu+1)}\int_{0}^{1}t^{\alpha-\mu}(qt;q)_{\beta+\mu}p_n(t;q^{\alpha-\mu},q^{\beta+\mu}|q)F_m(t)\,d_qt,
  \end{gathered}
  \]
  where
  \[F_m(t)=\frac{1}{(qt;q)_{\beta+\mu}}\int_{qt}^{1}u^{\mu-1}(qu;q)_{\beta}(qt/u;q)_{\mu-1}p_m(u;q^{\alpha},q^{\beta}|q))\,d_qu.\]
  Set \[G_m(t):=F_m(t)-\frac{\Gamma_q(\mu)\Gamma_q(\alpha-\mu+1)}{\Gamma_q(\alpha+1)}p_m(t;q^{\alpha-\mu},q^{\beta+\mu}|q)\]
  Therefore,
  \[
  \int_{0}^{1}w_{\alpha,\beta}(u)p_n(u;q^{\alpha-\mu},q^{\beta+\mu}|q)G_m(u)\,d_qu=0\]

  for all $n\in \mathbb{N}$. From the completeness of the little $q$-Jacobi polynomials $p_n(x;q^{\alpha},q^{\beta}|q)$ in the weighted space $L_q^2\left(A_q^*, w_{\alpha,\beta}\right)$, $w_{\alpha,\beta}(x)=x^{\alpha}(qx;q)_{\beta}$ we obtain
\[ F_m(u)=\dfrac{\Gamma_q(\mu)\Gamma_q(\alpha-\mu+1)}{\Gamma_q(\alpha+1)}\dfrac{c_m(\alpha,\beta)}{c_m(\alpha-\mu,\beta+\mu)}p_m(u;q^{\alpha-\mu},q^{\beta+\mu}|q).\]
  \end{proof}

\begin{cor}\label{cor:2}
  \[ \begin{gathered}D_{q,1^-}^{\mu}\left((qt;q)_{\beta+\mu}p_m(t;q^{\alpha-\mu},q^{\beta+\mu}|q)\right)\\
  =\Scale[.95]{q^{-m\mu}\frac{\Gamma_q(\mu+\beta+m+1)\Gamma_q(\alpha+m+1)\Gamma_q(\alpha-\mu+1)}{\Gamma_q(\beta+m+1)\Gamma_q(\alpha-\mu+1+m)\Gamma_q(\alpha+1)}
  (qt;q)_{\beta}p_m(t;q^{\alpha},q^{\beta}|q)}.
  \end{gathered}
  \]
\end{cor}

\begin{proof}
The proof follows form \eqref{CR1} and Lemma \ref{lem:2}.

\end{proof}

\begin{thm}For $0<\mu<1$, and $\beta>-1$, the   functions
\[\phi_n(x)=x^{\mu}p_n(x;q^{\mu},q^{\beta}|q),\quad n\in\mathbb{N}_0\]
are eigenfunctions of  the qFSLP\eqref{EX1}-- \eqref{BC5} associated to
the eigenvalues
\[\lambda_n:=q^{-n\mu}\dfrac{\Gamma_q(\mu+\beta+n+1)\Gamma_q(\mu+n+1)}{\Gamma_q(\beta+1+n) \Gamma_q(n+1)\Gamma_q(\mu+1)},\;n\in\mathbb{N}_0.\]
Moreover, the eigenfunctions  are unique up to a constant multiplier.
\end{thm}

\begin{proof}
In Corollary~\ref{cor:1},  set $\alpha=0$ and replace  $\beta$ by  $\beta+\mu$. This gives
\be\label{C:1}
{}^cD_{q,0^+}\left(x^{\mu}p_n(x;q^{\mu},q^{\beta}|q)\right)=\Gamma_q(\mu+1) p_n(x;1,q^{\beta+\mu}|q),
  \ee
In Corollary~\ref{cor:2}, set $\alpha=\mu$, we obtain
 \be\label{C:2} \begin{gathered}D_{q,1^-}^{\mu}\left((qx;q)_{\beta+\mu}p_n(x;1,q^{\beta+\mu}|q)\right)\\
  =q^{-n\mu}\dfrac{\Gamma_q(\mu+\beta+n+1)\Gamma_q(\mu+n+1)}{\Gamma_q(\beta+1+n) \Gamma_q(n+1)\Gamma_q(\mu+1)}
 (qx;q)_{\beta} p_n(x;q^{\mu},q^{\beta}|q).
  \end{gathered}\ee
  Combining \eqref{C:1} and \eqref{C:2} gives the required result.

\end{proof}

\section{Conclusion and future work}
This paper is the first paper introduces  fractional $q$-Sturm--Liouville problems. It proves the main  properties of the eigenvalues
and the eigenfunctions of the fractional Sturm--Liouville problem. We also   give a sufficient condition on an eigenvalue to have a unique eigenfunction, and a sufficient condition on the domain is give for the existence and uniqueness of eigenfunction. The discrete spectrum of a qFSLP is given.
This   work generalizes the study of integer Sturm--Liouville problem introduced by Annaby and Mansour in~\cite{AM1}. It is worth mentioning that a $q$-fractional variational calculus is developed and used in~\cite{Zeinab_FSL2} to prove that the qFSLp \eqref{EVP1} with the boundary condition $y(0)=y(a)=0$ has a countable set of real eigenvalues and associated  orthogonal eigenfunctions
 when  $\alpha>1/2$ and a  similar study for the fractional
Sturm--Liouville problem
\[{}^cD_{q,a^-}p(x)D_{q,0^+}^{\alpha}y(x)+r(x)y(x)=\lambda w_{\alpha}(x)y(x),\]
is in progress.


\begin{thebibliography}{10}
\expandafter\ifx\csname url\endcsname\relax
  \def\url#1{\texttt{#1}}\fi
\expandafter\ifx\csname urlprefix\endcsname\relax\def\urlprefix{URL }\fi
\expandafter\ifx\csname href\endcsname\relax
  \def\href#1#2{#2} \def\path#1{#1}\fi

\bibitem{AM1}
M.~Annaby, Z.~Mansour, Basic {Sturm}--{Liouville} {problems}, J. Phys.
  A:Math.~Gen. 38 (2005) 3775--3797, corrigendum, J. Phys. A:Math.~Gen.
  \textbf{39}(2006), 8747.

\bibitem{Sturm-Liouville}
C.~Sturm, J.~Liouville, Extrait d'un m\'{e}moire sur le d\'{e}veloppement des
  fonctions en s\`{e}ries dont les diff\'{e}rents termes sont assujettis \`{a}
  satisfaire \`{a} une m\^{e}me \'{e}quation diff\'{e}rentielle lin\'{e}aire{,}
  contenant un param\`{e}tre variable, J. Math. Pures Appl. 2.

\bibitem{STL}
W.~Amrein, A.~Hinz, D.~Pearson (Eds.), Sturm--{Liouville} Theorey: {Past} and
  {Present}, Birkh\"{a}user Verlag, Basel, 2005.

\bibitem{KOM}
M.~Klimek, T.~Odzijewicz, A.~Malinowska, Variational methods for the fractional
  {Sturm}--{Liouville} problem, J. Math. Anal. Appl. 416 (2014) 402--426.

\bibitem{Klimek-Agrawal}
M.~Klimek, O.~P. Agrawal, Fractional {Sturm}-{Liouville} problem, Comput. Math.
  Appl. 66~(5) (2013) 795--–812.

\bibitem{FSL1}
M.~Zayernouri, G.~Karniadakis, Fractional {Sturm}--{Liouville} eigen-problems:
  theory and numerical approximation, J. Comput. Phys. 252 (2013) 495--–517.

\bibitem{FSL2}
M.~Hajji, Q.~Al-Mdallal, F.~Allan, An efficient algorithm for solving
  higher-order fractional {Sturm}--{Liouville} eigenvalue problems, J. Comput.
  Phys. 272 (2014) 550--–558.

\bibitem{FSL3}
A.~Ansari, On finite fractional {Sturm}--{Liouville} transforms, Integral
  Transforms Spec. Funct. 26~(1) (2015) 51--–64.

\bibitem{FSL4}
M.~Zayernouri, M.~Ainsworth, G.~Karniadakis, Tempered fractional
  {Sturm}--{Liouville} eigenproblems, SIAM J. Sci. Comput. 37~(4) (2015)
  A1777--–A1800.

\bibitem{Lavagno}
A.~Lavagno, Basic-deformed quantum mechanics, Rep. Math. Phys. 64~(1-2) (2009)
  79--91.

\bibitem{Abreu005}
L.~Abreu, Sampling theory associated with $q$-difference equations of the
  {Sturm}-{Liouville} type, J. Phys. A 38~(48) (2005) 10311--–10319.

\bibitem{AnBuIs}
M.~Annaby, J.~Bustoz, M.~Ismail, On sampling theory and basic
  {Sturm}-{Liouville} systems, J. Comput. Appl. Math. 206 (2007) 73--85.

\bibitem{AMI}
M.~Annaby, Z.~Mansour, I.~Soliman, $q$-{Titchmarsh}--{Weyl} theory: series
  expansion, Nagoya Math. J. 205 (2012) 67--118.

\bibitem{Nemri-Fitouhi}
A.~Nemri, A.~Fitouhi, Polynomial expansions for solutions of wave equation in
  quantum calculus, Matematiche 65~(1) (2010) 73--82.

\bibitem{Abreu007}
L.~Abreu, Real {Paley}-{Wiener} theorems for the {Koornwinder}-{Swarttouw}
  q-{Hankel} transform, J. Math. Anal. Appl. 334~(1) (2007) 223--–231.

\bibitem{GR}
G.~Gasper, M.~Rahman, Basic {Hypergeometric} {Series}, Cambridge university
  Press, Cambridge, 2004.

\bibitem{AMbook}
M.~H. Annaby, Z.~S. Mansour, $q$-{Fractional} {Calculus} and {Equations}, Vol.
  2056 of Lecture Notes in Mathematics, Springer, Berlin, 2012.

\bibitem{Agr}
R.~Agarwal, Certain fractional $q$-integrals and $q$-derivatives, Proc. Camb.
  Phil. Soc. 66 (1969) 365--370.

\bibitem{Rajkovic09}
P.~Rajkovi\'{c}, S.~Marinkovi\'{c}, M.~Stankovi\'{c}, A generalization of the
  concept of $q$-fractional integrals, Acta Math. Sin. (Engl. Ser.) 25~(10)
  (2009) 1635--–1646.

\bibitem{KS}
H.~Koelink, R.~Swarttouw, On the zeros of the {Hahn}-{Exton} $q$-{Bessel}
  function and associated $q-${Lommel} polynomials, J. Math. Anal. Appl.
  186~(3) (1994) 690--710.

\bibitem{Zeinab_FSL2}
Z.~S. Mansour, Variational methods for fractional $q$-{Sturm}--{Liouville}
  problems, submitted (2016).

\end{thebibliography}
\end{document}